\documentclass[11pt,a4paper]{amsart}
\usepackage{amsmath, amssymb, amsthm, amsfonts}
\usepackage[colorlinks = true, citecolor = blue]{hyperref}

\theoremstyle{plain}
\newtheorem{theorem}{Theorem}[section]
\newtheorem{corollary}[theorem]{Corollary}
\newtheorem{lemma}[theorem]{Lemma}
\newtheorem{proposition}[theorem]{Proposition}
\newtheorem{fact}[theorem]{Fact}
\newtheorem*{theorem*}{Theorem}
\newcounter{MainCorollaryCounter}
\newtheorem{MainCorollary}[MainCorollaryCounter]{Corollary}

\newcounter{MainTheoremCounter}
\newtheorem{MainTheorem}[MainTheoremCounter]{Theorem}

\theoremstyle{definition}
\newtheorem{definition}[theorem]{Definition}

\newtheorem*{setup}{Setup}

\newtheorem{example}[theorem]{Example}

\newtheorem*{conjecture*}{Conjecture}

\newcommand{\fieldF}{F}

\newcommand{\mut}{\mathbb{M}(U,\tau)}
\newcommand{\pstar}{P^*}
\newcommand{\modu}[1]{\overline{#1}}
\newcommand{\Ubar}{\modu{U}}
\newcommand{\Hbar}{\modu{H}}
\newcommand{\textdef}[1]{\textit{#1}}
 
\newcommand{\setdiff}{-}
\newcommand{\hp}{hereditarily proper}

\DeclareMathOperator{\rk}{rk}  
\DeclareMathOperator{\dc}{d} 
\DeclareMathOperator{\pr}{pr} 
\DeclareMathOperator{\mr}{m} 
\DeclareMathOperator{\symm}{Sym} 
\DeclareMathOperator{\gl}{GL} 
\DeclareMathOperator{\psl}{PSL} 
\DeclareMathOperator{\ssl}{SL}

\DeclareMathOperator{\aut}{Aut}
\DeclareMathOperator{\emorph}{End} 
 
\DeclareMathOperator{\Mouf}{\mathbb{M}}

\begin{document}
\title[Moufang sets of finite Morley rank of odd type]{Moufang sets of finite Morley rank of odd type}
\author{Joshua Wiscons}
\address{Mathematisches Institut\\
Universit\"at M\"unster\\
Einsteinstrasse 62\\
 48149 M\"unster, Germany}
\email{wiscons@math.uni-muenster.de}
\thanks{This material is based upon work supported by the National Science Foundation under grant No. OISE-1064446.}
\begin{abstract}
We show that for a wide class of groups of finite Morley rank the presence of a split $BN$-pair of Tits rank $1$ forces the group to be of the form $\psl_2$ and the $BN$-pair to be standard. Our approach is via the theory of Moufang sets. Specifically, we investigate infinite and so-called \hp{} Moufang sets of finite Morley rank in the case where the little projective group has no infinite elementary abelian $2$-subgroups and show that all such Moufang sets are standard (and thus associated to $\psl_2(\fieldF)$ for $\fieldF$ an algebraically closed field of characteristic not $2$) provided the Hua subgroups are nilpotent. Further, we prove that the same conclusion can be reached whenever the Hua subgroups are  $L$-groups and the root groups are not simple. 
\end{abstract}
\maketitle

\section{Introduction}

Moufang sets were introduced by Jacques Tits in \cite{TiJ92} and capture the essence of split $BN$-pairs of Tits rank $1$. The classification of the finite Moufang sets was completed by Christoph Hering, William Kantor, and Gary Seitz in \cite{HKS72}, and there is now an effort to better understand the infinite ones. 

Moufang sets fall into one of two cases, \textdef{proper} and \textdef{not proper}, based on the permutation group that they describe; a proper Moufang set is one for which the associated group is \textbf{not} sharply $2$-transitive. The structure of infinite sharply $2$-transitive groups is poorly understood, but those with abelian point-stabilizers, i.e. those associated to nonproper Moufang sets with abelian root groups, are easily shown to be of a single flavor: one-dimensional affine groups over a field. With an eye towards classifying the Moufang sets with abelian root groups,  the recent work on  Moufang sets is focused on understanding those that are proper. An important example of such a Moufang set is the one associated to $\psl_2(F)$ for $F$ a field, see Example~\ref{exam.MF} below. This Moufang set is denoted $\Mouf(\fieldF)$, and a similar construction produces a canonical Moufang set, denoted $\Mouf(J)$, for any quadratic Jordan division algebra $J$. It is conjectured that every infinite proper Moufang set with abelian root groups is of the form $\Mouf(J)$ for some quadratic Jordan division algebra $J$.  When considering the rather restrictive class of Moufang sets of finite Morley rank, this conjecture \emph{seems} to have a positive answer, and here, all of the Moufang sets \emph{appear} to arise, as expected, from a field, see \cite{DMeTe08}, \cite{WiJ10},  \cite{WiJ11}. We now investigate if it is possible for Moufang sets of finite Morley rank to have nonabelian root groups. This is conjectured not to happen.

\begin{conjecture*}[see {\cite[Question B.17]{BoNe94}}]
If $\Mouf$ is an infinite proper Moufang set of finite Morley rank, then $\Mouf \cong \Mouf(\fieldF)$ for $\fieldF$ an algebraically closed field, and, hence, the little projective group is isomorphic to $\psl_2(F)$.
\end{conjecture*}

Groups of finite Morley rank are groups equipped with a model-theoretic notion of dimension, and it was conjectured by Gregory Cherlin and Boris Zil'ber in the late 70's that every infinite simple group of finite Morley rank is in fact an algebraic group over an algebraically closed field. As every simple algebraic group over an algebraically closed field has a $BN$-pair, it is natural to try to understand the groups of finite Morley rank with a  $BN$-pair. In \cite{KTVM99}, Linus Kramer, Katrin Tent, and Hendrik van Maldeghem showed that the simple groups of finite Morley rank with a $BN$-pair of Tits rank at least $3$ are all algebraic. Additionally, there are partial results when the Tits rank is $2$; a survey of these may be found in \cite{TeK02}. Thus, we find ourselves curious about the case of Tits rank $1$, and in this way, Moufang sets of finite Morley rank fit naturally into the investigation of the Cherlin-Zil'ber  Conjecture.  

As the Sylow $2$-subgroups of a group of finite Morley rank are conjugate, the analysis of groups of finite Morley rank can be broken into four cases based on the structure of the connected component $S^\circ$ of a Sylow $2$-subgroup $S$. These cases, with names corresponding to the characteristic of a possible interpretable field, are

\begin{description}
\item[Degenerate] $S^\circ$ is trivial,
\item[Even] $S^\circ$ is nontrivial, nilpotent, and of bounded exponent ($S^\circ$ is $2$-unipotent),
\item[Odd] $S^\circ$ is nontrivial, divisible, and abelian ($S^\circ$ is a $2$-torus), and
\item[Mixed] $S^\circ$ contains a nontrivial $2$-unipotent subgroup and a nontrivial $2$-torus.
\end{description}

In \cite{ABC08}, Tuna Alt{\i}nel, Alexandre Borovik, and Gregory Cherlin proved that there are no infinite simple groups of finite Morley rank of mixed type and that those of even type are algebraic. Not too much is known about degenerate groups, but there is a growing body of knowledge about groups of odd type. We will focus on Moufang sets whose little projective group has odd type, and as the little projective group of an infinite Moufang set of finite Morley rank is always connected and $2$-transitive, the following fact shows that we are only excluding the even type case. 

\begin{fact}[{\cite[Lemma~5.8]{BoCh08}}]
Let $G$ be a definably primitive and generically $2$-transitive permutation group of finite Morley rank. Then $G^\circ$ is either of odd or of even type. 
\end{fact}

Our main reference for groups of finite Morley rank will be  \cite{ABC08}; \cite{PoB87} and \cite{BoNe94} provide excellent introductions to the theory as well.

We now give the central definitions; the main results will follow.

\begin{definition}
For a set $X$ with $|X| \ge 3$ and a collection of groups $\{U_x : x\in X\}$ with each $U_x \le \mbox{Sym}(X)$, we say that $(X,\{U_x : x\in X\})$ is a \textdef{Moufang set} if for $G := \langle U_x : x\in X \rangle$ the following conditions hold:
\begin{enumerate}
\item each $U_x$ fixes $x$ and acts regularly on $X\setdiff \{x\}$,
\item $\{U_x : x\in X\}$ is a conjugacy class of subgroups in $G$.
\end{enumerate}
We call $G$ the \textdef{little projective group} of the Moufang set, and each $U_x$ for $x\in X$ is called a \textdef{root group}. The Moufang set is \textdef{proper} if the action of the little projective group on $X$ is \emph{not} sharply $2$-transitive. The $2$-point stabilizers are called the \textdef{Hua subgroups}. 
\end{definition}

Note that the little projective group of a Moufang set always acts $2$-transitively on $X$. Thus, the Hua subgroups are all conjugate as are the $1$-point stabilizers and the root groups.

\begin{definition}
We will say that a Moufang set $(X,\{U_x : x\in X\})$ with little projective group $G$ is \textdef{interpretable} in a structure if the root groups, $X$, $G$, and the action of $G$ on $X$ are all interpretable in the structure. Now we define a \textdef{Moufang set of finite Morley rank} to be a Moufang set interpretable in a structure of finite Morley rank. Finally, we say that a Moufang set of finite Morley rank has \textdef{odd type}, respectively \textdef{even type}, if the little projective group has odd type, respectively even type.
\end{definition}

It should be noted that in the previous works on Moufang sets of finite Morley rank the root groups were not assumed to be interpretable. However, all of the papers addressed the case of abelian root groups where one can easily show that the interpretability of $X$, $G$, and the action of $G$ on $X$ forces the root groups to be interpretable as well. Also, the interpretability of $X$ and the action of $G$ on $X$ follows from the interpretability of $G$ and the root groups and thus could be omitted from the definition. 

We have two main results. The first treats Moufang sets whose Hua subgroups are nilpotent, allowing the root groups to have an arbitrary structure. The second result is complementary. It addresses Moufang sets with nonsimple root groups, this time placing few restrictions on the Hua subgroup. The fact that neither theorem assumes the  Moufang sets to be special (as defined in Section~\ref{sec.MS}) is noteworthy as much of the theory of infinite Moufang sets is built around this hypothesis. 

\begin{MainTheorem}\label{thm.A}
Let $\Mouf$ be an infinite \hp{} Moufang set of finite Morley rank of odd type. If the Hua subgroups are nilpotent, then $\Mouf \cong \Mouf(\fieldF)$ for $\fieldF$ an algebraically closed field. 
\end{MainTheorem}

The property of being \hp{} is defined in Section~\ref{sec.PStar} and addresses the issue that in proper, but not necessarily special, Moufang sets one may encounter root subgroups for which the induced Moufang set is no longer proper. This is conjectured not to happen outside of a few known small examples, see \cite[Conjecture 2]{SeY09}. We make the observation in Corollary~\ref{cor.HeredProper} that a Moufang set whose little projective group has odd type is automatically \hp{} if the root groups are without involutions. 

Our second theorem requires another definition. A group of finite Morley rank is an \textdef{$L$-group} if \emph{every} infinite definable simple section of odd type is isomorphic to an algebraic group over an algebraically closed field; by \textdef{definable section} we mean a quotient of a definable subgroup by one of its definable normal subgroups. The study of $L$-groups (and $L^*$-groups) is aimed at a classification of the simple groups of odd type that does not rely on knowledge of degenerate type sections.

\begin{MainTheorem}\label{thm.B}
Let $\Mouf$ be an infinite \hp{} Moufang set of finite Morley rank of odd type. If the Hua subgroups are $L$-groups, then $\Mouf \cong \Mouf(\fieldF)$ for $\fieldF$ an algebraically closed field \textbf{unless} both of the following occur:
\begin{itemize}
\item the root groups are simple of degenerate type, and 
\item the Hua subgroups have Pr\"ufer $2$-rank at least $2$.
\end{itemize}
\end{MainTheorem}

The conclusion of the theorem states that the root groups are either very nice, i.e. isomorphic to the additive group of a field, or very wild, i.e.  simple degenerate type groups with many nontrivial involutory automorphisms. In the context of Moufang sets, the latter alternative seems exceptionally wild as it has been conjectured that the root groups in \emph{any} Moufang set should be nilpotent, see \cite{DST08}.

Let us also mention that the hypothesis that the Hua subgroups be $L$-groups is only needed when addressing the case of abelian root groups. That is, our approach to Theorem~\ref{thm.B} is to show that, in a minimal counterexample, avoiding either of the final two configurations implies that the root groups are abelian; we then appeal to the results of \cite{DMeTe08} and \cite{WiJ11} where the latter utilizes the $L$-hypothesis. 

Theorem~\ref{thm.B} immediately yields the following corollary on $K^*$-groups; a group of finite Morley rank is a \textdef{$K^*$-group} if  \emph{every proper} infinite definable simple section (regardless of type) is isomorphic to an algebraic group over an algebraically closed field.

\begin{MainCorollary}
Let $G$ be a simple odd type $K^*$-group of finite Morley rank. If $(B,N,U)$ is a split $BN$-pair of Tits rank $1$ for  $G$ with $U$ definable and without involutions, then $G$ is isomorphic to $\psl_2(F)$ for $\fieldF$ an algebraically closed field, and $(B,N,U)$ is the standard $BN$-pair for $G$.
\end{MainCorollary}

This corollary says something nontrivial even for simple algebraic groups over algebraically closed fields; note that such groups are $K^*$-groups of finite Morley rank by \cite[Th\'eor\`eme~4.13]{PoB87}. Of course, one is not worried about classifying the simple algebraic groups, but the question of whether or not there are nonstandard $BN$-pairs in these groups is interesting. For  some recent and rather general results in this vein (but with the restriction that ``U'' is nilpotent), one may see \cite{PrG11}.

\section{Moufang sets}\label{sec.MS}
We briefly give some background on Moufang sets. For a more thorough introduction to the general theory, the reader is encouraged to see \cite{DMeSe09}. Specifics regarding a context of finite Morley rank can be found in \cite{WiJ11}. 

As shown in \cite{DMeWe06}, every Moufang set can be constructed as follows from a, not necessarily abelian, group $(U;+,-,0)$ and a permutation $\tau \in \symm(U \cup \{\infty\})$ interchanging $0$ and $\infty$. Define $\mut$ to be $(X,\{U_x : x\in X\})$ where each $U_x$ is a subgroup of $\symm(X)$ defined as follows:
\begin{enumerate}
\item for each $u\in U$, $\alpha_u$ is the permutation of $X$ that fixes $\infty$ and sends each $v\in U$ to $v+u$;
\item $U_\infty := \{\alpha_u : u \in U\}$;
\item $U_0 := U_\infty^\tau$; and
\item $U_u := U_0^{\alpha_u}$ for each $u\in U^*$.
\end{enumerate}
Given such a construction, there is, for each $a\in U^*$, a unique element of $U_0\alpha_aU_0$ interchanging $0$ and $\infty$, and it is referred to as $\mu_a$.  It is a theorem of \cite{DMeWe06} that $\mut$ will be a Moufang set precisely when the set $\{\tau\mu_a : a \in U^*\}$ is contained in $\aut(U)$. When $\mut$ is a Moufang set, the pointwise stabilizer of $0$ and $\infty$ in the little projective group is called \textbf{the} \textdef{Hua subgroup}, and it is generated by the set $\{\mu_a\mu_b : a,b \in U^*\}$. We said above that each $2$-point stabilizer is called \textbf{a} Hua subgroup, but there should not be any confusion as all $2$-point stabilizers are conjugate. We will frequently use the fact that for any $h$ in the Hua subgroup and any $a \in U^*$ we have that $\mu_a^h = \mu_{ah}$. Finally, a Moufang set $\mut$ is called \textdef{special} if the action of $\tau$ on $U^*$ commutes with inversion.

\begin{fact}[{\cite[Main~Theorem]{SeY09}}]\label{fact.Dichotomy}
If $\mut$ is a proper Moufang set with abelian root groups, then $\mut$ is special.
\end{fact}

\begin{fact}[{\cite[Theorem~1$\cdot$2]{SeWe08}}]\label{fact.Irreducible}
If $\mut$ is a special Moufang set, then either $U$ is an elementary abelian $2$-group, or the Hua subgroup acts irreducibly on $U$.
\end{fact}

Let us give an example of a (special) Moufang set. When we work in a setting of finite Morley rank, it will be the only one we have in mind.

\begin{example}\label{exam.MF}
Let $\fieldF$ be a field. We define $\Mouf(\fieldF)$ to be $\mut$ where $U:= \fieldF^+$ and $\tau$ is the permutation of $X:=U\cup\{\infty\}$ swapping $0$ and $\infty$ and sending each $x \in \fieldF^*$ to $-x^{-1}$. Then $\Mouf(\fieldF)$ is a (special) Moufang set, see \cite[Example~3.1]{DMeWe06}, with little projective group $\psl_2(\fieldF)$.
\end{example}

As most of our work in the present paper does not assume $U$ to be abelian, we will, sadly, not be blessed with the rather cozy feeling that comes with being special. However, finite Morley rank is quite adequate compensation. We will use the following fact often and without specific reference.

\begin{fact}[{\cite[Proposition~2.3]{WiJ11}}]
Let $\Mouf$ be an infinite Moufang set of finite Morley rank. The little projective group as well as all $1$-point stabilizers, all $2$-point stabilizers, and all root groups are (definable and) connected.
\end{fact}

\subsection{Root subgroups}

Root subgroups are the essential ingredient of an inductive approach to any theorem about Moufang sets. 

\begin{definition}
Let $\mut$ be a Moufang set. A \textdef{root subgroup} of $U$ is a subgroup $V \le U$ such that there exists some $v\in V^*$ with $V^*\mu_v = V^*$. 
\end{definition}

A root subgroup gives rise to a  sub-Moufang set $\Mouf(V,\rho)$ where $\rho$ is the restriction of $\mu_v$ to $V\cup\{\infty\}$ for some $v\in V^*$. This Moufang set will be the same for every $v\in V^*$ and will be called the \textdef{Moufang set induced by $V$}. Extremely useful is the fact that every subgroup of the form $C_U(h)$ for $h$ in the Hua subgroup is a root subgroup. Also, one can show that the root subgroup relation is transitive: if $V$ is a root subgroup of $U$ and $W$ is a root subgroup of $V$ (in the Moufang set induced by $V$), then $W$ is a root subgroup of $U$.

When pulling back information from an induced Moufang to the original, we will utilize the subgroups $G(V)$ and $H(V)$ defined below as well as the fact that the induced little projective group can be identified with $G(V)/C_{G(V)}(V)$.

\begin{definition}
Let $\mut$ be a Moufang set with little projective group $G$ and Hua subgroup $H$. For $V$ a root subgroup of $U$, define
\begin{itemize}
\item $G(V) := \langle \alpha_v , \mu_v : v \in V^* \rangle \le G$, and
\item $H(V) := \langle \mu_v\mu_w : v,w \in V^* \rangle \le H$.
\end{itemize}
\end{definition}

\begin{fact}[{\cite[Lemma~3.7]{WiJ11}}]\label{Fact_InterpretInduced}
Let $\mut$ be an infinite Moufang set of finite Morley rank. If $V$ is an infinite definable root subgroup of $U$, then $V$ induces an interpretable Moufang set with $G(V)$ definable and connected. 
\end{fact}

\section{$\pstar$-Moufang sets}\label{sec.PStar}
We now begin our analysis of $\pstar$-Moufang sets; this will be refined for the odd type setting in Section~\ref{sec.PStarOdd}. The results of this section are drawn from the author's thesis and are, for the most part, just mild generalizations of those in \cite{WiJ11}.

\begin{definition}
Let $\Mouf:= \mut$ be an infinite Moufang set of finite Morley rank.
\begin{itemize}
\item $\Mouf$ is \textdef{projective} if $\Mouf \cong \Mouf(\fieldF)$ for some algebraically closed field $\fieldF$.
\item A root subgroup is \textdef{projective} if it induces a projective Moufang set.
\item $\Mouf$ is a \textdef{$\pstar$-Moufang set} if every infinite proper definable root subgroup of $U$ is projective.
\end{itemize}
\end{definition}

$\pstar$-Moufang sets are the critical objects of study when entertaining the idea that some proper Moufang sets of finite Morley rank may not be projective. To make this explicit, we restrict to so-called \hp{} Moufang sets.

\begin{definition}
Let $\mut$ be an infinite Moufang set of finite Morley rank. Then $\mut$ is said to be \textdef{\hp{}} if it is proper and every infinite definable root subgroup of $U$ induces a proper Moufang set.
\end{definition}

Since the root subgroup relation is transitive, every infinite definable root subgroup of a \hp{} Moufang set induces a Moufang set that is also \hp{}. Special Moufang sets are the primary examples of \hp{} Moufang sets. Indeed, it is a fact that special Moufang sets are necessarily proper whenever the root groups have at least $3$ elements. As the property of being special easily passes to root subgroups, we see that infinite special Moufang sets of finite Morley rank are \hp{}. Of course, proper $\pstar$-Moufang sets are also \hp{}. 

\begin{lemma}\label{lem.pstarExists}
If $\mut$ is an infinite \hp{} Moufang set of finite Morley rank that is not projective, then $U$ contains an infinite definable root subgroup that induces a proper $\pstar$-Moufang set that is also not projective.
\end{lemma}
\begin{proof}
Assume that $\mut$ is \hp{} and not projective. Let $V$ be an infinite definable nonprojective root subgroup that is of minimal rank among all such root subgroups. Clearly $V$ induces a proper $\pstar$-Moufang set that is not projective.
\end{proof}

Let us fix some notation for the remainder of the present section. 

\begin{setup}
$\mut$ is an infinite proper Moufang set of finite Morley rank. Let $G$ be the little projective group, $H$ the Hua subgroup, and $X:=U \cup \{\infty\}$.
\end{setup}

\subsection{Projective root subgroups}

The following proposition is the starting point for all of our analysis. Note that for a group $B$, we write $B=A_1*A_2$ if $A_1$ and $A_2$ are commuting subgroups that generate $B$, i.e. $B$ is the \textdef{central product} of $A_1$ and $A_2$.

\begin{proposition}\label{prop.NHV}
Let $V$ be an infinite definable projective root subgroup of $U$. Then $V$ induces a Moufang set with little projective group isomorphic to $\psl_2(\fieldF)$ for $\fieldF$ an algebraically closed field, and for $Y:= V\cup \{\infty\}$, the following are true:
\begin{enumerate}
\item $G(V) \cong  \ssl_2(\fieldF)$ or $\psl_2(\fieldF)$,
\item $H(V)$ is definable and isomorphic to $\fieldF^\times$,
\item $H(V)$  generates an interpretable field, isomorphic to $\fieldF$, in $\emorph(V)$,
\item $N_H(V) = H(V) * C_H(V)$ with $H(V) \cap C_H(V) = Z(G(V))$, and
\item $N_G(Y) = G(V) * C_H(V)$ with $G(V) \cap C_H(V) = Z(G(V))$.
\end{enumerate}
\end{proposition}
\begin{proof}
The proof of this proposition is nearly identical to that of \cite[Proposition~3.11]{WiJ11}. We identify the induced little projective group with $G(V)/C_{G(V)}(V)$. As $V$ is projective, there is an algebraically closed field $\fieldF$ such that $G(V)/C_{G(V)}(V)$ acting on $Y$ is isomorphic to $\psl_2(\fieldF)$ acting naturally on the projective line over $\fieldF$. 

By Fact~\ref{Fact_InterpretInduced}, $G(V)$ is definable and connected, and \cite[Lemma~3.2(3)]{SeY08} says that $G(V)$ has center $C_{G(V)}(V)$. Additionally, the proof of \cite[Lemma~3.2(4)]{SeY08} only requires that the \emph{induced} Moufang set be special, a hypothesis we certainly meet, so  $G(V)$ is perfect. Now \cite[Lemma~3.10]{WiJ11} applies, and $G(V) \cong  \ssl_2(\fieldF)$ or $\psl_2(\fieldF)$. 

For the next two items, we first note that the pointwise stabilizer of $0$ and $\infty$ in $G(V)$, namely $H(V)C_{G(V)}(V)$, is definable and isomorphic to the stabilizer of $0$ and $\infty$ in $\ssl_2(\fieldF)$ or $\psl_2(\fieldF)$. The third item is now clear, and to complete the second item, we show that  $C_{G(V)}(V) \le H(V)$. The nontrivial case is when $G(V) \cong  \ssl_2(\fieldF)$. Fix a $v\in V^*$. Then, $\mu_v$ is in $G(V)_{\{0,\infty\}} \setdiff G(V)_{0,\infty}$. When $\ssl_2(\fieldF)$ acts on the projective line, each element that swaps $0$ and $\infty$ squares to the central involution. Thus, $\langle \mu_v^2 \rangle = Z(G(V)) = C_{G(V)}(V)$, and the second item is complete as $\mu_v^2$ is in $H(V)$.

We now give the structure of $N_H(V)$. Notice that $H(V)$ is normal in $N_H(V)$ and is centralized by $C_H(V)$.  We have already mentioned that $G(V)$ has center $C_{G(V)}(V)$, so we see that $H(V) \cap C_H(V) = G(V) \cap C_H(V) = Z(G(V))$, since $H(V)$ contains $Z(G(V))$. It remains to show that  $N_H(V) = H(V)C_H(V)$, which we do exactly as in \cite[Proposition~3.11]{WiJ11}. Now, $H(V)$ generates an interpretable field in $\emorph(V)$, say $E$, and $N_H(V)$ acts on $E$ as a group of field automorphisms. By \cite[I,~Lemma~4.5]{ABC08}, $N_H(V)$ acts $E$-linearly on $V$. As $V$ is $1$-dimensional over $E$ with $H(V)$ inducing all of $E^\times$, we find that $N_H(V) = H(V)C_H(V)$. This completes the fourth item, and the final point now follows from \cite[Lemma~3.2(2)]{SeY08}.
\end{proof}

The proposition yields the following two important corollaries.

\begin{corollary}\label{cor.RootAndChar}
If $V$ is an infinite definable projective root subgroup of $U$ and $A$ is an $H(V)$-invariant subgroup of $U$, then either $V \le A$ or $V \cap A = 0$.
\end{corollary}
\begin{proof}
$H(V)$ acts transitively on $V^*$ and normalizes $V\cap A$.
\end{proof}

\begin{corollary}\label{cor.MaxRoot}
If $V \le W$ are two infinite definable projective root subgroups of $U$, then $V = W$.
\end{corollary}
\begin{proof}
In light of Proposition~\ref{prop.NHV}, the proof of this corollary is identical to that of \cite[Corollary~3.13]{WiJ11}.
\end{proof}

\subsection{$H$-invariant projective root subgroups}

We now collect a couple of items regarding $H$-invariant root subgroups.  We begin with a lemma that is a trivial, but important, consequence of Proposition~\ref{prop.NHV}.

\begin{lemma}\label{lem.HNormalRoot}
If $V$ is an infinite definable $H$-invariant projective root subgroup of $U$, then $H = H(V)*C_H^\circ(V)$, and in particular, $H(V) \le Z(H)$.
\end{lemma}
\begin{proof}
Everything follows directly from Proposition~\ref{prop.NHV} upon remembering that $H$ is connected.
\end{proof}

The next proposition says that in a  $\pstar$-setting either $H$ is isomorphic to the multiplicative group of a field or $H$ is close to acting irreducibly on $U$, in some weak sense. Both conclusions are nice approximations to the situation for $\Mouf(\fieldF)$.

\begin{proposition}\label{prop.TwoHNormalRoot}
Assume that $\mut$ is a $\pstar$-Moufang set. If $U$ contains distinct infinite proper definable $H$-invariant root subgroups $V$ and $W$, then $H=H(V)=H(W)$.
\end{proposition}
\begin{proof}
We begin by showing that $H$ is abelian. By Lemma~\ref{lem.HNormalRoot}, we have that $H = H(V)*C^\circ_H(V)= H(W)*C^\circ_H(W)$ with both $H(V)$ and $H(W)$ abelian. We claim that $C_H(V) \cap C_H(W) = 1$. If not, there is an $h\in H^*$ with $C_U(h)$ containing both $V$ and $W$. Since, by assumption, $C_U(h)$, $V$, and $W$ are all projective, Corollary~\ref{cor.MaxRoot} implies that $C_U(h)=V=W$, which is a contradiction. We conclude that $C_H(V) \cap C_H(W) = 1$, so $H$ embeds into $H/C_H(V) \times H/C_H(W)$. Thus, $H$ is abelian.

To prove the proposition, we must show that $C_H(V)$ and $C_H(W)$ are both finite. Set $A = C_H(V)C_H(W)$. We claim that $A$ is central in $G_{\{0,\infty\}}$. To see this, first recall that $G_{\{0,\infty\}}$ is generated by $H$, which is abelian, and any $\mu$-map. As $C_H(V)$ is centralized by every $\mu_v$ with $v\in V^*$, $C_H(V) \le Z(G_{\{0,\infty\}})$. Similarly, we find that $C_H(W) \le Z(G_{\{0,\infty\}})$, so $A \le Z(G_{\{0,\infty\}})$. We now show that $A$ has at most four elements. Let $a \in A$ be arbitrary. Then  $a \in hC_H(V)$ for some $h\in H(V)$. For any $v \in V^*$, we have that \[aC_H(V)= (aC_H(V))^{\mu_v} = (hC_H(V))^{\mu_v}  = h^{-1}C_H(V) = a^{-1}C_H(V), \] where the second to last equality uses that  $V$ is projective and the $\mu$-maps invert the Hua subgroup in $\psl_2$. Thus, $C_H(V)$ has index at most $2$ in $A$, and a similar argument shows that the same is true for $C_H(W)$. Since $C_H(V) \cap C_H(W) = 1$, we conclude that $A$ has order at most four.
\end{proof}

\section{When $G$ has odd type}\label{sec.OddType}
This short section begins our analysis of Moufang sets whose little projective group has odd type. As before, we  fix some notation. 

\begin{setup}
$\mut$ is an infinite (not necessarily proper) Moufang set of finite Morley rank. Let $G$ be the little projective group, $H$ the Hua subgroup, and $X:=U \cup \{\infty\}$. Further, assume that $G$ is of odd type.
\end{setup}

The following lemma makes use of a ``generosity argument.'' A definable subgroup of a group of finite Morley rank is said to be \textdef{generous} if the union of the conjugates of the subgroup is generic, i.e. has full rank in the group. A necessary condition for being generous is to have ``enough'' conjugates, and we call a subgroup \textdef{almost self-normalizing} if it has finite index in its normalizer. An important example of a generous subgroup in any connected group of finite Morley rank is the connected component of the centralizer of a decent torus, see \cite[IV, ~Lemma~1.14]{ABC08}. Recall that a \textdef{decent torus} is a divisible abelian group of finite Morley rank that is the definable hull of its torsion subgroup.

\begin{lemma}\label{lem.GenZassenhaus}
If $C_X(g)$ is finite for all $g\in G^*$, then $G_\infty$ has odd type. If, additionally, $H$ is nontrivial, then  $H$ has odd type as well.
\end{lemma}
\begin{proof}
Define $A$ to be $H$ if $H$ is nontrivial and $G_\infty$ otherwise. We show that $A$ has odd type.

Let $T$ be the definable hull of a nontrivial $2$-torus of $G$. Then $T$ is a decent torus, and $C^\circ_G(T)$ is generous in $G$. We now work to establish that $A$ is generous in $G$, and we begin by showing that $A$ is almost self-normalizing. Let $N := N^\circ_G(A)$. Then $N$ acts on $C_X(A)$, which is a nonempty finite set. As $N$ is connected,  $N$ fixes $C_X(A)$. Since $A$ is the full pointwise stabilizer of one or two points of $C_X(A)$, it must be that $N \le A$, so $A$ is almost self-normalizing. Now our assumption that $C_X(a)$ is finite for all $a\in A^*$ implies the second condition of \cite[IV,~Lemma~1.25]{ABC08}, so  $A$ is generous.

We conclude that $\bigcup_{g\in G}C^\circ_{G}(T^g)$ and $\bigcup_{g\in G}A^g$ have a nontrivial intersection, so $C^\circ_{G}(T^g) \cap A \neq 1$ for some $g \in G$. Choose a nontrivial $a\in C^\circ_{G}(T^g) \cap A$. Now, $T^g$ is a connected group acting on the finite set $C_X(a)$. As before, $T^g$ fixes $C_X(a)$, so $T^g\le A$. 
\end{proof}

The previous lemma allows us to extend \cite[Propostion~3.9]{WiJ11} to Moufang sets for which $U$ need not be abelian. The proof uses the following basic, but absolutely essential, fact about centralizers of involutions in groups of finite Morley rank.

\begin{fact}[{\cite[I,~Lemma~10.3]{ABC08}}]\label{fact.fixedPointsInvolution}
If a connected group of finite Morley rank has a definable involutory automorphism $\alpha$ with finitely many fixed points, then the group is abelian and $\alpha$ is inversion. 
\end{fact}

\begin{corollary}[Zassenhaus Moufang sets]\label{cor.FiniteFixed}
If $\mut$ is proper and $C_U(h)$ is finite for all $h\in H^*$, then $\mut \cong \Mouf(\fieldF)$ for some algebraically closed field $\fieldF$.
\end{corollary}
\begin{proof}
It suffices to show that $U$ is abelian and then appeal to \cite[Propostion~3.9]{WiJ11}, by way of Fact~\ref{fact.Dichotomy}. By Lemma~\ref{lem.GenZassenhaus},  $H$ contains an involution, and this involution acts on the connected group $U$ with finitely many fixed-points. Hence, $U$ is abelian.
\end{proof}

We will, on several occasions, make use of the following result  about degenerate type groups.

\begin{fact}[{\cite[Theorem~1]{BBC07}}]\label{fact.BBC}
A connected degenerate type group of finite Morley rank has no involutions.
\end{fact}

With this fact in hand, we now use Lemma~\ref{lem.GenZassenhaus} to expose another class of \hp{} Moufang sets. Recall that we are only considering when the little projective group has odd type.

\begin{corollary}\label{cor.HeredProper}
If $U$ has no involutions, then $\mut$ is \hp{}.
\end{corollary}
\begin{proof}
First, Lemma~\ref{lem.GenZassenhaus} forces $\mut$ to be proper. Let $V$ be an infinite definable root subgroup of $U$. As $G$ has odd type, the little projective group of the induced Moufang set has odd or degenerate type, but the latter case is ruled out by Fact~\ref{fact.BBC} and the observation that $2$-transitive groups of finite Morley rank contain involutions. Now assume that the induced Moufang set is not proper. Then the root groups of the induced Moufang set coincide with the $1$-point stabilizers in the induced little projective group. By Lemma~\ref{lem.GenZassenhaus}, the roots groups of the induced Moufang set have odd type. However, the root groups of the induced Moufang set are subgroups of the original root groups, and we have a contradiction.
\end{proof}

The next corollary of Lemma~\ref{lem.GenZassenhaus} shows that \hp{} Moufang sets (in odd type) always have nontrivial projective root subgroups.

\begin{corollary}\label{cor.minimalRootDicotomy}
If $\mut$ is proper, then every minimal infinite definable root subgroup of $U$ is either projective or induces a nonproper Moufang set. 
\end{corollary}
\begin{proof}
Assume that $\mut$ is proper. Let $V$ be a minimal infinite definable root subgroup, and let $\Mouf'$ be the Moufang set induced by $V$. Suppose that $\Mouf'$ is proper. It is not hard to see that Corollary~\ref{cor.FiniteFixed} applies to $\Mouf'$, so $\Mouf'$ is projective. Indeed, towards a contradiction, suppose that some nontrivial element in the Hua subgroup of $\Mouf'$ fixes infinitely many points of $V$. Then there is an $h \in H$ such that $C_V(h)$ is infinite and properly contained in $V$. As the intersection of root subgroups is again a root subgroup, $C_V(h) = V \cap C_U(h)$ is a root subgroup, and we have contradicted the minimality of $V$. We conclude that $\Mouf'$ is projective.
\end{proof}

We end this section with a final corollary of Lemma~\ref{lem.GenZassenhaus} showing that the structure of \hp{} Moufang sets in odd type resembles that of $\psl_2$ in characteristic not $2$. The result will be used often. Notice that the corollary also shows, in combination with Corollary~\ref{cor.HeredProper}, that (in odd type) $\mut$ is \hp{} if and only if $U$ has no involutions. 

\begin{corollary}\label{cor.UHstruct}
If $\mut$ is \hp{}, then $H$ has odd type, and $U$ contains no decent torus and, hence, no involutions.
\end{corollary}
\begin{proof}
Assume that $\mut$ is \hp{}, and let $V$ be any minimal infinite definable root subgroup. By Corollary~\ref{cor.minimalRootDicotomy}, $V$ is projective, and Proposition~\ref{prop.NHV} now shows that $H(V)$ contains a $2$-torus. 

To see that $U$ contains no decent torus, we argue by contradiction. Let $\mut$ be a counterexample for which $U$ is of minimal rank among all such counterexamples. Let $T\le U$ be a decent torus. We have already seen that $H$ has odd type, so let $S$ be a decent torus of $H$. By the conjugacy of maximal decent tori in $U\rtimes H$ (see \cite[IV,~Proposition~1.15]{ABC08} for example), $S$ centralizes some conjugate of $T$ which, of course, also lies in $U$. Thus, $W:= C_U(S)$ contains a decent torus. Now, $W$ induces a Moufang set of finite Morley rank $\Mouf'$ that is infinite and hereditarily proper. Further, the little projective group of $\Mouf'$ must contain infinite Sylow $2$-subgroups and is isomorphic to a section of $G$,  so it must have odd type. Since $U$ is connected, the rank of $W$ is less than the rank of $U$, and we have contradicted the minimality of our counterexample. 

Finally, $G$ has odd type, so it must be that the Sylow $2$-subgroups of $U$ are finite as $U$ contains no divisible torsion. Since $U$ is connected, we conclude that $U$ has trivial Sylow $2$-subgroups by Fact~\ref{fact.BBC}.
\end{proof}

\section{$\pstar$-Moufang sets in odd type}\label{sec.PStarOdd}
We now refine our analysis of $\pstar$-Moufang sets for the odd type setting. Of course, the goal, which is not achieved here, is to show that these Moufang sets are each isomorphic to $\Mouf(F)$ some algebraically closed field $F$. We do however show that these Moufang sets fall into one of two extreme cases. Specifically, this section will prove the following.

\begin{theorem}\label{thm.Dicotomy}
Let $\mut$ be an infinite proper $\pstar$-Moufang set of finite Morley rank of odd type. Then either
\begin{enumerate}
\item the root groups are abelian, or
\item the root groups are simple.
\end{enumerate}
Further, if the root groups are simple, then the Hua subgroup $H$ is nonnilpotent with $\pr_2 H = \mr_2 H = 2$, and $H$ leaves invariant at most one proper nontrivial definable subgroup of $U$ which, if it exists, is a root subgroup.
\end{theorem}

Regarding the notation, the Pr\"{u}fer $2$-rank of a group $A$, denoted $\pr_2 A$, is the maximal $n$ such that the group contains a subgroup isomorphic to $\left(\mathbb{Z}(2^\infty)\right)^n$ where $\mathbb{Z}(2^\infty) := \{x\in \mathbb{C} : x^{2^k} = 1 \text{ for some $k \in \mathbb{N}$} \}$. The (regular) $2$-rank of a group $A$, denoted $\mr_2 A$,  is the maximal rank of all elementary abelian $2$-subgroups of $A$. One always has $\pr_2 A \le \mr_2 A$. 

We  carry the following setup throughout the section.

\begin{setup}
$\mut$ is an infinite proper $\pstar$-Moufang set of finite Morley rank. Let $G$ be the little projective group, $H$ the Hua subgroup, and $X:=U \cup \{\infty\}$. Further, assume that $G$ is of odd type.
\end{setup}

\subsection{The structure of $H$}

We temporarily focus our attention on $H$. Our first proposition is extremely important as, among other things, it ensures that $H$ contains a Klein $4$-group whenever $U$ is not abelian. This will allow us to repeatedly exploit the following fact. 

\begin{fact}[{\cite[Proposition~9.1]{BBC07}}]\label{fact.FourGroup}
Let $K$ be a group generated by two distinct commuting involutions. If $K$ acts definably on a group of finite Morley rank $A$ that is without involutions, then $A = \langle C_A(x) : x\in K^* \rangle.$
\end{fact}

\begin{proposition}\label{prop.Prufer}
If $U$ is not abelian, then $\pr_2 H = \mr_2 H = 2$.
\end{proposition}
\begin{proof}
Assume $U$ is not abelian. We first show that $\pr_2 H \ge 2$. Let $i$ be an involution of $H$, and set $V:= C_U(i)$. Let $Y := V\cup\{\infty\}$. As $U$ is not abelian, $V$ is infinite (see Fact~\ref{fact.fixedPointsInvolution}), and hence projective. Let $T$ be the $2$-torus of $H(V)$, and recall that $T$ is central in $N_H(V)$. We consider two cases.

First, assume that $i\notin T$. By \cite[Theorem~3]{BuCh08}, $i$ is contained in some $2$-torus of $H$, say $S$. As $S$ centralizes $i$, $S < N_H(V)$, so $S$ commutes with $T$. Since $i\notin T$, $T$ and $S$ are distinct commuting $2$-tori, so $\pr_2 H \ge 2$.

Now we treat the case when $i \in T$. We first show that $N_H(V)$ and $N_G(Y)$ are connected. Let $\dc(T)$ be the definable hull of $T$. Notice that $C_H(\dc(T)) \le C_H(i) \le N_H(V) \le C_H(\dc(T))$. Thus,  $N_H(V) = C_H(\dc(T))$, so  \cite[Theorem~1]{AlBu08} shows that $N_H(V)$ is connected. By \cite[Lemma~3.2(2)]{SeY08}, $N_G(Y) = G(V)*N_H(V)$. Hence, $N_G(Y)$ is generated by connected groups and is therefore connected as well.

As $G$ has finite Morley rank,  there is some involution $\omega$ in $G_{\{0,\infty\}} \setdiff H$. Let $D$ be the definable hull of $\langle i\omega \rangle$. As $D < G_{\{0,\infty\}}$, $i$ and $\omega$ are not $D$-conjugate. By \cite[I,~Lemma~2.20]{ABC08}, there is an involution $k \in D$ (different from $i$ and $\omega$) that commutes with both $i$ and  $\omega$. If $k\in H$, then we may apply \cite[Theorem~3]{BuCh08} to find a $2$-torus of $N_H(V)$ containing $k$, hence different from $T$. Thus, we are done unless $k \in G_{\{0,\infty\}} \setdiff H$. 

The situation is now that $i \in T$, $k\in G_{\{0,\infty\}} \setdiff H$, and $[i,k]=1$. Since  $N_G(Y)$ is connected, Proposition~\ref{prop.NHV} tells us that $N_G(Y) = G(V)* C_H^\circ(V)$. As $k$ centralizes $i$, $k\in N_G(Y)$, so $k = gc$ for some $g\in G(V)$ and some $c\in C_H^\circ(V)$.  Notice that our assumption that $i\in H(V)\cap C_H(V)$ forces $G(V) \cong \ssl_2(F)$ for some algebraically closed field $F$. Since, $g \in G(V) \cap G_{\{0,\infty\}}$ and $g\notin H$, $g$ must have order $4$. As $k$ has order $2$, $c \neq 1$. Since $g$ and $c$ commute, $c$ has order dividing $4$, and we conclude that $C_H^\circ(V)$ contains an involution. Thus, $C_H^\circ(V)$ contains a $2$-torus by Fact~\ref{fact.BBC}, and $\pr_2 H\ge 2$.

In both cases, $\pr_2 H \ge 2$. We conclude by showing that $\mr_2 H < 3$. Towards a contradiction, suppose that $H$ has an elementary abelian $2$-subgroup $K$ of order $8$. Choose distinct involutions $a,b,c \in K$ such that $c \notin \langle a,b\rangle$. For each  $j \in K$, set $V_j = C_U(j)$. Now, $b$ and $c$ act on $V_a$. If $b$ centralizes $V_a$, then Corollary~\ref{cor.MaxRoot} forces $V_a = V_b$, and thus, Fact~\ref{fact.FourGroup} implies that $U = V_a$, a contradiction. Since $b$ does not centralize $V_a$,  Proposition~\ref{prop.NHV} tells us that $b$ inverts $V_a$. Similarly, $c$ inverts $V_a$. Thus, $bc$ centralizes $V_a$, so Corollary~\ref{cor.MaxRoot} implies that $V_a = V_{bc}$. By assumption, $a \neq bc$, so we may apply Fact~\ref{fact.FourGroup}  to again arrive at the contradiction that $V_a = U$. 
\end{proof}

We now give a couple of quick corollaries. The first will be refined below in Proposition~\ref{prop.OneHInvar}.

\begin{corollary}\label{cor.OneHInvarRoot}
$U$ has at most one infinite proper definable $H$-invariant root subgroup.
\end{corollary}
\begin{proof}
Assume that $U$ contains distinct infinite proper definable $H$-invariant root subgroups $V$ and $W$. Notice that $U$ is not abelian as otherwise Facts~\ref{fact.Dichotomy} and \ref{fact.Irreducible} would imply that $H$ acts irreducibly on $U$. Thus, $\pr_2 H \ge 2$. However, Proposition~\ref{prop.TwoHNormalRoot} says that $H=H(V)=H(W)$. Since $H(V)$ is isomorphic to the multiplicative group of a field, it must be that $\pr_2 H = 1$.
\end{proof}

\begin{corollary}\label{cor.NilpotentHua}
$H$ contains at most one central involution. In particular, if $H$ is nilpotent, $\mut \cong \Mouf(\fieldF)$ for $\fieldF$ an algebraically closed field.
\end{corollary}
\begin{proof}
Suppose that $H$ contains two distinct central involutions, and let $K$ be the group they generate. Notice that $C_U(x)$ is finite for at most one $x \in K^*$, as otherwise, two distinct involutions of $H$ would invert $U$ implying that their product fixes $U$. Thus, there are distinct $x,y \in K^*$ such that $C_U(x)$ and $C_U(y)$ are infinite proper definable $H$-invariant root subgroups. By Corollary~\ref{cor.OneHInvarRoot}, $C_U(x) = C_U(y)$. But then the fixed-point spaces of every $x \in K^*$ coincide (using Corollary~\ref{cor.MaxRoot}), which contradicts Fact~\ref{fact.FourGroup}. Thus, $H$ contains at most one central involution.

Now consider when $H$ is nilpotent. We claim that every $2$-torus of $H$ is central. Indeed, the torsion subgroup of $H$ has a unique Sylow $2$-subgroup (simply because $H$ is nilpotent), and the connected component of this Sylow $2$-subgroup must be a (or rather \emph{the}) maximal $2$-torus of $H$ since $H$ has odd type. As this maximal $2$-torus is normal in $H$, it is in fact central in $H$ by \cite[I,~Corollary~5.25]{ABC08}, so every $2$-torus of $H$ is central, as claimed. Since $H$ contains at most one central involution, Proposition~\ref{prop.Prufer} implies that $U$ is abelian, and the rest now follows from Fact~\ref{fact.Dichotomy} and \cite{DMeTe08}.
\end{proof}

\subsection{The structure of $U$}

We now work to limit the structure of $U$. We begin with a couple of lemmas.

\begin{lemma}\label{lem.HNormalRootOdd}
Suppose that $V$ is an infinite proper definable $H$-invariant root subgroup of $U$. Let $i$ be the unique involution of $H(V)$. The following are true.
\begin{enumerate}
\item $G(V) \cong \ssl_2(F)$ for some algebraically closed field $F$.
\item $C_H(V)$ contains a unique involution, namely $i$.
\item \label{lem.HNormalRootOdd.Free} $H(V) \setdiff \langle i\rangle$ acts freely on $U^*$.
\end{enumerate}
\end{lemma}
\begin{proof}
By Facts~\ref{fact.Dichotomy} and \ref{fact.Irreducible}, $U$ is not abelian, so $C_U(i)$ is infinite. By Lemma~\ref{lem.HNormalRoot}, $i$ is central in $H$, so $C_U(i)$ is again an infinite proper definable $H$-invariant root subgroup of $U$. By Corollary~\ref{cor.OneHInvarRoot}, it must be that $V=C_U(i)$, so $i \in H(V)\cap C_H(V)$. Using Proposition~\ref{prop.NHV}, we find that $G(V) \cong \ssl_2(F)$ for some algebraically closed field $F$.

Now let $j$ be any involution in $C_H(V)$; we will show that $i=j$. Suppose $i \neq j$, and set $K = \langle i,j\rangle$. By Corollary~\ref{cor.MaxRoot}, $C_U(j) = V$, but then Fact~\ref{fact.FourGroup} implies that $V=U$, which is a contradiction. 

Finally, suppose that some $h\in H(V)$ has a proper nontrivial fixed-point space. Then $C_U(h)$ is an $H$-invariant root subgroup of $U$. We claim that $C_U(h)$ must be infinite. Since $U$ is not abelian, $H$ has distinct commuting involutions. Thus, Fact~\ref{fact.FourGroup}, applied to $C_U(h)$, shows that $C_U(h)\cap C_U(k) \neq 1$ for some involution $k\in H$. Notice that $C_U(k)$ is infinite since $U$ is not abelian. By Corollary~\ref{cor.RootAndChar}, $C_U(h)$ contains $C_U(k)$. Thus, $C_U(h)$ is infinite, so Corollary~\ref{cor.OneHInvarRoot} forces $V = C_U(h)$. As $i$ is the only nontrivial element of $H(V)$ fixing $V$, $h=i$.
\end{proof}

We now make a Morley rank calculation important to our eventual proof that ``$U$ not simple implies $U$ abelian.''

\begin{lemma}\label{lem.RankBound}
If $U$ has an infinite proper definable $H$-invariant root subgroup, then $\rk U \ge \rk H$.
\end{lemma}
\begin{proof}
Suppose that $U$ has a proper definable $H$-invariant root subgroup $V$, and, towards a contradiction, assume that $\rk U < \rk H$. We will show that the Moufang set is special which, by Fact~\ref{fact.Irreducible}, is a contradiction. We will use a local approach to showing that $\mut$ is special: $\mut$ is special if and only if $(-a)\mu_a = a$ for all $a \in U^*$, see \cite[Lemma~7.1.4]{DMeSe09}.

Let $a\in U^*$ be arbitrary. Since $\rk U < \rk H$, $C_H(a)$ is nontrivial. Let $h$ be a nontrivial element in $C_H(a)$, and set $W:= C_U(h)$. By Lemma~\ref{lem.HNormalRoot}, $H(V)$ is central in $H$, so $H(V)$ acts on $W$. Now Lemma~\ref{lem.HNormalRootOdd}\eqref{lem.HNormalRootOdd.Free} implies that $W$ is infinite. By the $\pstar$-hypothesis, $W$ is projective, so $(-a)\mu_a = a$.
\end{proof}

Fact~\ref{fact.Irreducible} says that the Hua subgroup of a special Moufang set almost always acts irreducibly on $U$. The next proposition says that in our setting, which does not assume that the Moufang set is special, either $H$ acts definably irreducibly on $U$ or there is exactly one proper nontrivial definable $H$-invariant subgroup of $U$. Further, in the latter case, the exceptional subgroup is in fact a root subgroup.

\begin{proposition}\label{prop.OneHInvar}
If $A$ is a proper nontrivial definable $H$-invariant subgroup of $U$, then $H$ has a unique central involution $i$ and $A = C_U(i)$. In particular, $A$ is a root subgroup, so $A$ is infinite, connected, and abelian.
\end{proposition}
\begin{proof}
First note that Facts~\ref{fact.Dichotomy} and \ref{fact.Irreducible} imply that $U$ is not abelian. Let $K$ be any subgroup of $H$ generated by two distinct commuting involutions. For each $x \in K^*$ set $V_x = C_U(x)$. As $U$ is not abelian, each $V_x$ is infinite. By Corollary~\ref{cor.RootAndChar}, it must be that $V_x \cap A = 0$ or $V_x \le A$ for each $x \in K^*$. Apply Fact~\ref{fact.FourGroup} to $A$ to see that $V_i \le A$ for some $i \in K^*$. Further, Fact~\ref{fact.FourGroup} applied to $U$  ensures that $V_k \cap A = 0$ for some $k \in K^*$, since $A$ is proper. Let $j$ be the remaining involution in $K$. We have that either $A = V_i$ or $A=\langle V_i,V_j \rangle$. 

If $A = V_i$, then $A$ is a root subgroup, and Lemma~\ref{lem.HNormalRootOdd} tells us that $i\in H(V_i)$. By Lemma~\ref{lem.HNormalRoot}, $i$ is central in $H$ and, hence, the unique central involution by Corollary~\ref{cor.NilpotentHua}. 

We now work to rule out the possibility that $A=\langle V_i,V_j \rangle$. Towards a contradiction, assume that $A=\langle V_i,V_j \rangle$. The first step is to show that $A$ is connected and abelian. We observed above that $V_k \cap A = 0$, so it must be that $V_i\cap V_j = 0$. Since $V_i$ and $V_j$ are root subgroups, and hence connected, Zil'ber's Indecomposability Theorem implies that $A$ is connected. Further, $k$ acts on the connected group $A$ without nontrivial fixed points, so $A$ is inverted by $k$. Thus $A = V_i \times V_j$ is connected and abelian. 

Next, we show that $H$ has an infinite center and acts faithfully and irreducibly on $A$. By Corollary~\ref{cor.MaxRoot}, $C_H(V_i) \cap C_H(V_j) = 1$, so $H$ acts faithfully on $A$. As $k$ inverts $A$, $k$ is in the center of $H$, so $V_k$ is $H$-invariant. Set $V := V_k$ and $T := H(V)$. By Lemma~\ref{lem.HNormalRoot}, $T \le Z(H)$. To see that $H$ acts irreducibly on $A$, first notice that Fact~\ref{fact.FourGroup} and Corollary~\ref{cor.RootAndChar} would force any proper nontrivial definable $H$-invariant subgroup of $A$ to be equal to $V_i$ or $V_j$. This situation would violate Corollary~\ref{cor.OneHInvarRoot}, so $H$ must act definably irreducibly on $A$. Thus, $H$ acts irreducibly on $A$ by \cite[I,~Lemma~11.3]{ABC08}.

We now appeal to \cite[I,~Proposition~ 4.11]{ABC08} to see that $T$ generates an interpretable algebraically closed field $F$ in $\emorph(A)$, $A$ is a vector space over $F$, and $H$ embeds into $\gl(A)$. Next, we show that $A$ is $2$-dimensional over $F$. Note that $V_i$ is a subspace of $A$; it is an eigenspace of $i$. Now, $H(V_i)$ is abelian, so $H(V_i)$ normalizes some $1$-dimensional subspace of $V_i$. Of course, $H(V_i)$ acts transitively on $V_i$, so it must be that $V_i$ is $1$-dimensional over $F$. The same argument works for $V_j$, so $A$ is $2$-dimensional. 

By \cite[Th\'eor\`eme~4]{PoB01a}, $H'$ is either solvable or equal to $\ssl_2(F)$. Note that $H' \le C_H(V)$ and that $C_H(V)$ must act freely on $A$ by Corollary~\ref{cor.MaxRoot}. In particular, $C_H(V)$ contains no unipotent elements, in the algebraic sense, so $H'$ must be solvable. Thus, $H$ is connected and solvable, so it is contained in a Borel subgroup of $\gl(A)$. As such, $H$ normalizes some $1$-dimensional subspace of $A$ contradicting the fact that $H$ acts irreducibly on $A$.
\end{proof}

The following corollary completes the proof of Theorem~\ref{thm.Dicotomy}; the proof will use a variant of Fact~\ref{fact.fixedPointsInvolution}.

\begin{fact}[{\cite[I,~Lemma~10.4]{ABC08}}]\label{fact.decompositionInvolution}
If a group of finite Morley rank that is without involutions has a definable involutory automorphism $\alpha$, then $\alpha$ inverts each element of some transversal for $C_G(\alpha)$ in $G$. 
\end{fact}

\begin{corollary}\label{cor.SimpleAbelian}
Either $U$ is simple, or $U$ is abelian.
\end{corollary}
\begin{proof}
Assume that $U$ is not abelian. We first treat the case when $U$ is definably characteristically simple. Let $F^*(U)$ be the generalized Fitting subgroup of $U$. As the Fitting subgroup must have a trivial center, $U$ embeds into $\aut(F^*(U))$. Thus, it cannot be that $F^*(U) = 0$, so $F^*(U) = U$. Since $U$ is not abelian, $U$ equals the layer of $U$. As $H$ is connected, it acts trivially on the finite set of components of $U$, so each component is $H$-invariant. By Proposition~\ref{prop.OneHInvar}, there cannot be more than one component. Thus, $U$ is quasisimple, but then it must be that $U$ is simple.

To complete the proof, we show that $U$ must be definably characteristically simple. Suppose not, and let $V$ be a proper nontrivial definable characteristic subgroup of $U$. Then $H$ has a unique central involution $k$, and $V = C_U(k)$. By Fact~\ref{fact.decompositionInvolution}, there is a transversal for $U/V$ consisting of elements inverted by $k$, so $k$ inverts $U/V$. Also, $U/V$ is connected since $U$ is connected.

We now derive a contradiction in much the same way as in Proposition~\ref{prop.OneHInvar} except that we now must work with the quotient $U/V$. Set $T := H(V)$. Then, $H = T*C_H^\circ(V)$ with $T \le Z(H)$. Let $K$ be the group generated by $k$ and any other involution of $H$ (using Proposition~\ref{prop.Prufer}). Choose $i \in K \setdiff \langle k \rangle$, and set $j = ik$. Also, let $V_i = C_U(i)$ and $V_j = C_U(j)$. Observe that $V_i$ and $V_j$ are each different from $V$ and $T$-invariant, so Corollary~\ref{cor.RootAndChar} shows that both intersect $V$ trivially. Thus they embed into  $U/V$.

We use bar notation to denote passage to the quotient $\Ubar := U/V$. Set $M := C_H(\Ubar)$; we also use bar notation for the quotient $\Hbar:=H/M$. Recall that $\Ubar$ is abelian. Now, by Fact~\ref{fact.FourGroup}, $\Ubar = \modu{V_i}\cdot\modu{V_j}$.  Since $i$ fixes $V_i$ and inverts $V_j$, it must be that $\modu{V_i}\cap\modu{V_j}$ is trivial as otherwise $\Ubar$, and hence $U$, would have involutions, see \cite[I,~Proposition~2.18]{ABC08}. Thus, $\Ubar = \modu{V_i}\times\modu{V_j}$. As Proposition~\ref{prop.OneHInvar} implies that $V$ is the only proper nontrivial definable $H$-invariant subgroup of $U$, $H$ must act definably irreducibly on $A$. Thus, $H$ acts irreducibly on $A$ by \cite[I,~Lemma~11.3]{ABC08}. Also, $T$ acts freely on $V_i$, so $T$ embeds into $\Hbar$. Thus, $T$ generates an interpretable algebraically closed field $F$ in $\emorph(\Ubar)$, and $\modu{H}$ embeds into $\gl(\Ubar)$. Further, we find, as in Proposition~\ref{prop.OneHInvar}, that $\Ubar$ is $2$-dimensional over $F$.

We now show that $\modu{H}$ is solvable, and here the approach differs slightly from the proof of Proposition~\ref{prop.OneHInvar}. By \cite[Th\'eor\`eme~4]{PoB01a}, $\modu{H}' = \modu{H'}$ is either solvable or equal to $\ssl_2(F)$. We claim that $\modu{H'}$ cannot be $\ssl_2(F)$ because it is too small. First recall that $H' \le C_H(V)$. Now, let $f$ be the rank of $F$. By Lemma~\ref{lem.RankBound}, $\rk H \le \rk U$. Now, $\rk H = \rk T + \rk C_H(V)$ and $\rk U = \rk V + 2f$. Since $\rk T = \rk V$, we have that $\rk C_H(V) \le 2f$. Thus $\rk \modu{H'} \le 2f$, so $\modu{H'}$ cannot be $\ssl_2(F)$. Thus, $\modu{H}$ is solvable. As in the proof of Proposition~\ref{prop.OneHInvar}, this yields a contradiction.
\end{proof}

\section{Proofs of the main theorems}\label{sec.Theorem}
The main theorems  follow rather quickly from Theorem~\ref{thm.Dicotomy}. 

\begin{proof}[Proof of Theorem~\ref{thm.A}]
Suppose that Theorem~\ref{thm.A} is false, and let $\Mouf = \mut$ be a counterexample. By Lemma~\ref{lem.pstarExists}, there is an infinite definable root subgroup $V\le U$ that induces a proper $\pstar$-Moufang set $\Mouf'$ that is also not projective. However, the induced Moufang set also satisfies the hypotheses of the theorem, and this contradicts Theorem~\ref{thm.Dicotomy} or, more specifically, Corollary~\ref{cor.NilpotentHua}.
\end{proof}

For the proof of Theorem~\ref{thm.B}, we begin with a lemma.

\begin{lemma}
Let $\Mouf$ be an infinite \hp{} Moufang set of finite Morley rank of odd type. Assume that the Hua subgroups are $L$-groups. Then \emph{either} of the following imply that $\mut \cong \Mouf(\fieldF)$ for $\fieldF$ an algebraically closed field:
\begin{enumerate}
\item the root groups are solvable, or
\item the Hua subgroup has Pr\"ufer $2$-rank at most $1$.
\end{enumerate}
\end{lemma}
\begin{proof}
Suppose that the lemma is false. Let $\Mouf = \mut$ be a nonprojective Moufang set satisfying the hypotheses of the lemma and one of the final two conditions. As before, we may use Lemma~\ref{lem.pstarExists} to find an infinite definable root subgroup $V\le U$ that induces a proper $\pstar$-Moufang set $\Mouf'$ that is not projective. The induced Moufang set also satisfies the hypotheses of the lemma and one of the final two conditions, so Theorem~\ref{thm.Dicotomy} implies that $V$ is abelian. By Fact~\ref{fact.Dichotomy}, $\Mouf'$ is also special. Now we may appeal to  \cite[Theorem~2.1]{DMeTe08} when $V$ is torsion free and \cite[Theorem~3.1]{WiJ11} when $V$ is elementary abelian, the latter requiring the $L$-hypothesis, to see that $\Mouf'$ is projective, which is a contradiction.  
\end{proof}

\begin{proof}[Proof of Theorem~\ref{thm.B}]
Suppose that the theorem is false. By the previous lemma, there must exist an infinite \hp{} Moufang set of finite Morley rank of odd type whose Hua subgroups are $L$-groups and whose root groups are not solvable but also not simple. Let $\Mouf = \mut$ be a Moufang set satisfying these hypotheses whose root groups have minimal rank among all such Moufang sets. Notice that the previous lemma also implies that the Hua subgroup, $H$, has distinct commuting involutions $i$ and $j$. In what follows, note that we are \emph{not} in a $\pstar$-setting. However, we do know, by Corollary~\ref{cor.UHstruct}, that $U$ has no involutions.

We first claim that $U$ has a proper nontrivial definable normal subgroup that is $H$-invariant. If not, we would find that $U$ is equal to its generalized Fitting subgroup. Then, as $U$ and $H$ are connected, each component of $U$ would be normalized by $U$ and $H$, and we would find that $U$ is quasisimple, and hence simple. This would contradict the nonsimplicity of the root groups, so it must be that $U$  has a proper nontrivial definable normal $H$-invariant subgroup $A$. 

Let $K := \langle i,j \rangle$. For each $x\in K^*$, set $V_x:=C_U(x)$. We now claim that for each $x\in K^*$ either $V_x \cap A$ is  trivial or equal to $V_x$. Towards a contradiction assume that $A_x := V_x\cap A$ is a proper nontrivial subgroup of $V_x$. Let $\Mouf_x$ be the Moufang set induced by $V_x$. Then $\Mouf_x$ is also an infinite \hp{} Moufang set of finite Morley rank of odd type whose Hua subgroups are $L$-groups. Further, $A_x$ is normalized by the induced Hua subgroup, so $\Mouf_x$ is certainly not projective. By the previous lemma, $V_x$ is not solvable, and the presence of $A_x$ shows that $V_x$ is not simple. We conclude that $\Mouf_x$ violates the minimality of $\Mouf$, so it must be that $V_x \cap A$ is either trivial or all of $V_x$. Further, we can use this observation to see that $V_x \le A$ for some $x\in K^*$. Indeed, if $V_i$ and $V_j$ intersect $A$ trivially, then both $i$ and $j$ invert $A$, so in this case, $k=ij$ would centralize $A$. Thus, $V_x \cap A$ is nontrivial for some $x\in K^*$, so $V_x \le A$ for some $x\in K^*$. Fix $x\in K^*$ with $V_x \le A$.

Now, if $A$ is finite, then $x$  acts on the connected group $U$ with finitely many fixed points, so $x$ inverts $U$. This is a contradiction, so $A$ must be infinite. Since $V_x \le A$, we find that $x$ inverts $U/A$ by Fact~\ref{fact.decompositionInvolution}, so $U/A$ is abelian. We now show that $A$ is also abelian. Applying Fact~\ref{fact.FourGroup} to $U$ it must be that $V_y$ is not contained in $A$ for some $y\in K^*$. As before, this means that $V_y \cap A$ is trivial, so we may again use Fact~\ref{fact.decompositionInvolution} to see that $y$ inverts $A$. Thus, $A$ is abelian, and this contradicts the nonsolvability of $U$.
\end{proof}

\section*{Acknowledgments}
The author would like to thank Katrin Tent and Tuna Alt{\i}nel for several valuable suggestions and comments that lead to a refinement of both the content and the presentation of the paper. The author is also quite grateful to the referee for a very careful reading of the article and many helpful recommendations.

\bibliographystyle{alpha}
\bibliography{WisconsBib}
\end{document}